\documentclass[12pt, a4paper, reqno]{amsart}

\usepackage{enumerate}
\usepackage{amssymb}
\usepackage{color}
\usepackage{graphics}
\usepackage{epsfig}

\newtheorem{lemma}{Lemma}[section]

\newtheorem{theorem}[lemma]{Theorem}

\newtheorem{remark}[lemma]{Remark}

\def\erf{\operatorname{erf}}

\newcommand{\R}{\mathbb{R}}

\newcommand{\N}{\mathbb{N}}
\newcommand{\Z}{\mathbb{Z}}
\newcommand{\T}{\mathcal{T}}

\author[J. D. Garc\'{\i}a-Salda\~{n}a]{Johanna D. Garc\'{\i}a-Salda\~{n}a}
\address{Dept. de Matem\`{a}tiques \\
Universitat Aut\`{o}noma de Barcelona \\ Edifici C. 08193
Bellaterra, Barcelona. Spain} \email{johanna@mat.uab.cat}

\author[A. Gasull]{Armengol Gasull}
\address{Dept. de Matem\`{a}tiques.
Universitat Aut\`{o}noma de Barcelona. Edifici C. 08193 Bellaterra,
Barcelona. Spain} \email{gasull@mat.uab.cat}

\date{}
\dedicatory{} \commby{}

\begin{document}

\title[Weak periodic solutions and the Harmonic Balance Method]
{Weak periodic solutions of  $x\ddot{x}+1=0$ and the Harmonic
Balance Method}
\begin{abstract}
We prove that the differential equation $x\ddot{x}+1=0$ has
continuous weak periodic solutions and compute their periods.
 Then, we use the Harmonic Balance Method until order six to approach
 these periods and to illustrate how the
sharpness of the method increases with the order. Our computations
rely on the Gr\"{o}bner basis method.
\end{abstract}

\maketitle

\section{Introduction and main results}
The nonlinear differential equation
\begin{equation}\label{xx+1}
x\ddot{x}+1=0,
\end{equation}
appears in the modeling of certain phenomena in plasma physics \cite{Ac-Sq}. In~\cite{Mi}, Mickens
calculates the period of its periodic orbits and also uses the $N$-th order Harmonic Balance
Method (HBM), for $N=1,2$, to obtain  approximations of these periodic solutions and of their
corresponding periods. Strictly speaking, it can be easily seen that neither
equation~\eqref{xx+1}, nor its associated system
\begin{equation}\label{sis1}
\left\{\begin{array}{lll}
\dot{x}=y,\\
\dot{y}=-\frac{1}{x},
\end{array}\right.
\end{equation}
which is singular at $x=0$, have periodic solutions. Our first
result gives two different interpretations of Mickens' computation
of the period.  The first one in terms of weak (or generalized)
solutions. In this work a weak solution will be a function
satisfying the differential equation~\eqref{xx+1} on an open and
dense set, but being of class ${\mathcal C}^0$ at some isolated
points. The second one, as the limit, when $k$ tends to zero, of the
period of actual periodic solutions of the extended planar
differential system
\begin{equation}\label{sis2}
\left\{\begin{array}{lll} \dot{x}=y,\\\dot{y}=-\frac{x}{x^{2}+k^{2}},
\end{array}\right.
\end{equation}
which, for $k\ne0,$ has a global center at the origin.

\begin{theorem}\label{main1}
\begin{enumerate}[(i)]
\item For the initial conditions $x(0)=A\ne0,\, \dot x(0)=0,$ the differential equation~\eqref{xx+1}
has a weak $\mathcal{C}^0$-periodic solution with period
$T(A)=2\sqrt{2\pi}A.$
\item  Let  $T(A;k)$ be the period of the periodic orbit of system~\eqref{sis2} with initial
conditions $x(0)=A,\, y(0)=0.$ Then
\[
\quad T(A;k)=4\,A\int_0^1\frac{ds}{\sqrt{\ln{\left(\frac{A^2+k^2}{A^2 s^2+k^2}\right)}}}\]
and
\[
\lim_{k\to0} T(A;k)=4A\int_0^1\frac1{\sqrt{-2\ln s}}ds=2\sqrt{2\pi}A=T(A).
\]
\end{enumerate}
\end{theorem}

 Recall that the  $N$-th order HBM consists in
approximating the solutions of differential equations by truncated
Fourier series with $N$ harmonics and an unknown frequency; see for
instance~\cite{Mi-lib,Mic-book2} or Section~\ref{hbm} for a short
overview of the method. In~\cite[p. 180]{Mic-book2} the author asks
for techniques for dealing analytically with the $N$-th order HBM,
for $N\ge 3$. In~\cite{Ga-Ga2} it is shown how resultants can be
used when $N=3$. Here we utilize  a more powerful tool, the
computation of Gr\"obner basis (\cite[Ch. 5]{cox}), for going
further in the obtention of approximations of the function $T(A)$
introduced in Theorem~\ref{main1}.

Notice that equation~\eqref{xx+1} is equivalent to the family of differential equations
\begin{equation}\label{eq xm}
x^{m+1}\ddot{x}+x^m=0,
\end{equation}
for any $m\in\N\cup\{0\}$. Hence it is natural to approach the period function,
$$T(A)=2\sqrt{2\pi}A\approx 5.0132A,$$
by the periods of the trigonometric polynomials obtained applying
the $N$-th order HBM to~\eqref{eq xm}. Next theorem gives our
results for $N\le6.$ Here $[a]$ denotes the integer part of  $a.$

\begin{theorem}\label{main2} Let $\T_N(A;m)$ be the period of the truncated Fourier series
obtained applying the $N$-th order HBM to equation~\eqref{eq xm}. It
holds:
\begin{enumerate}[(i)]
\item For all $m\in\N\cup\{0\}$, \begin{equation}\label{PFfmic Sin}
\T_1(A;m)=2\pi \sqrt{\frac{2[\frac{m+1}2]+1}{2[\frac{m+1}2]+2}}\,A.
\end{equation}
\item For $m=0,$
\begin{align*}
 &\T_1(A;0)=\sqrt{2}\,\pi A\approx 4.4428 A,\qquad\quad\phantom{m} \T_2(A;0)=(\sqrt{218}/9)\,\pi A\approx 5.1539 A,\\
&\T_3(A;0)={\textstyle\frac{13810534\,\pi A}
{3\sqrt{5494790257313+115642506449\sqrt{715}}}}\approx4.9353 A,\quad \T_4(A;0)\approx 5.0455
A,\\
&\T_5(A;0)\approx 4.9841 A,\qquad\quad\qquad\qquad\qquad\qquad\qquad\phantom{mm} \T_6(A;0)\approx
5.0260 A,
\end{align*}
\item For $m=1,$
\begin{align*}
&\T_1(A;1)=\sqrt{3}\,\pi A\approx 5.4414 A, &&\T_2(A;1)\approx 5.2733 A,\\
&\T_3(A;1)\approx 5.1476 A, &&\T_4(A;1)\approx 5.1186 A.
\end{align*}
\item For $m=2,$
\begin{align*}
&\T_1(A;2)=\sqrt{3}\,\pi A\approx 5.4414 A, &&\T_2(A;2)\approx 5.2724 A,\\
&\T_3(A;2)\approx 5.1417 A. 
\end{align*}
\end{enumerate}
Moreover, the approximate values appearing  above  are roots of
given polynomials with integer coefficients. Whereby the Sturm
sequences approach can be used to get them with any desirable
precision.
\end{theorem}

Notice  that the values $\T_1(A;m),$ for $m\in\{0,1,2\}$ given in
items (ii), (iii) and (iv), respectively, are already computed in
item (i). We only explicite them to clarify the reading.

Observe that the comparison of \eqref{PFfmic Sin} with the value
$T(A)$ given in Theorem~\ref{main1} shows that when $N=1$ the best
approximations  of $T(A)$ happen when $m\in\{1,2\}$. For this reason
we have applied the HBM for $N\le6$ and $m\le2$ to elucidate which
of the approaches is better. In the Table~\ref{tperror} we will
compare the percentage of the relative errors
\[
e_N(m)=100\left|\frac{ \T_N(A;m)-T(A)}{T(A)}\right|.
\]
 The best approximation that we have found corresponds to $\T_6(A;0).$ Our computers have had
problems to get the Gr\"{o}bner basis needed to fill the gaps of the table.

\begin{table}[h]\label{Ta1}
\begin{center}
$$\begin{array}{|c||c|c|c|}
\hline \quad e_N(m)\quad &\quad m=0\quad & \quad m=1
\quad& \quad m=2 \quad\\
\hline \hline
N=1 & 11.38\%& 8.54\%& 8.54\% \\
\hline
N=2 & 2.80\% & 5.19\%& 5.17\%  \\
\hline
N=3 &1.55\% & 2.68\%& 2.56\%  \\
\hline N=4 & 0.64\% & 2.10\% & -\\
\hline N=5 & 0.58\% & - & -\\
\hline N=6 & 0.25\% & - & -\\\hline
\end{array}$$
\caption{Percentage of relative errors $e_N(m)$.}\label{tperror}
\end{center}
\end{table}
The paper is organized as follows. Theorem~\ref{main1} is proved in Section~\ref{solus}.  In
Section \ref{hbm}  we describe the $N$-th order HBM adapted to our purposes. Finally, in
Section~\ref{sec sys} we use this method to demonstrate Theorem~\ref{main2}.

\section{Proof of Theorem~\ref{main1}.}\label{solus}

\noindent $(i)$ We start proving that the solution of \eqref{xx+1} with initial conditions
$x(0)=A$, $\dot x(0)=0$  and for $t\in\left(-\frac{\sqrt{2\pi}}{2}A,\frac{\sqrt{2\pi}}{2}A\right)$
is
 \begin{equation}\label{solx}
 x(t)=\phi_0(t):=A e^{-\left(\erf^{-1}\left(\frac{{2}\,t}{\sqrt{2\pi}
A}\right)\right)^2},
\end{equation}
 where $\erf^{-1}$ is the inverse of the error function
\[
\erf(z)=\frac 2{\sqrt{\pi}}\int_0^z e^{-s^2}\,ds.
\]
Notice that $\lim_{t\rightarrow
\pm\frac{\sqrt{2\pi}}{2}}\phi_0(t)=0$ and $\lim_{t\rightarrow
\pm\frac{\sqrt{2\pi}}{2}}\phi_0^\prime(t)=\mp\infty$. To obtain
\eqref{solx}, observe that from system~\eqref{sis1} we arrive at the
simple differential equation
$$
\frac{dx}{dy}=-xy,
$$
which has  separable variables  and can be solved by integration.  The particular solution that
passes by the point $(x,y)=(A,0)$ is
\begin{equation}\label{part sol}
x=Ae^{-y^2/2}.
\end{equation}
Combining \eqref{sis1} and  \eqref{part sol} we obtain
$$
\frac{dy}{dt}=-\frac{e^{y^2/2}}{A},
$$
again a separable equation. It has the solution
\begin{equation}\label{soly}
y(t)=-\sqrt{2}\,\erf^{-1}\left(\frac{2\,t}{\sqrt{2\pi} A}\right),
\end{equation}
which is well defined for
$t\in\left(-\frac{\sqrt{2\pi}}{2}A,\frac{\sqrt{2\pi}}{2}A\right)$
since $\erf^{-1}(\cdot)$ is defined in $(-1,1)$. Finally, by
replacing $y(t)$ in \eqref{part sol} we obtain \eqref{solx}, as we
wanted to prove.

By using $x(t)$ and $y(t)$ given by \eqref{solx} and \eqref{soly},
respectively, or using \eqref{part sol}, we can draw the phase
portrait of \eqref{sis1} which, as we can see in Figure
\ref{figura1}.(b), is symmetric with respect to both axes.  Notice
that its orbits do not cross the $y$-axis, which is a singular locus
for the associated vector field. Moreover, the solutions of
\eqref{xx+1} are not periodic (see Figure \ref{figura1}.(a)), and
the transit time of $x(t)$ from $x=A$ to $x=0$ is
$\sqrt{2\pi}\,A/2$.

\begin{figure}[h]
\begin{center}
\begin{tabular}{cc}
\epsfig{file=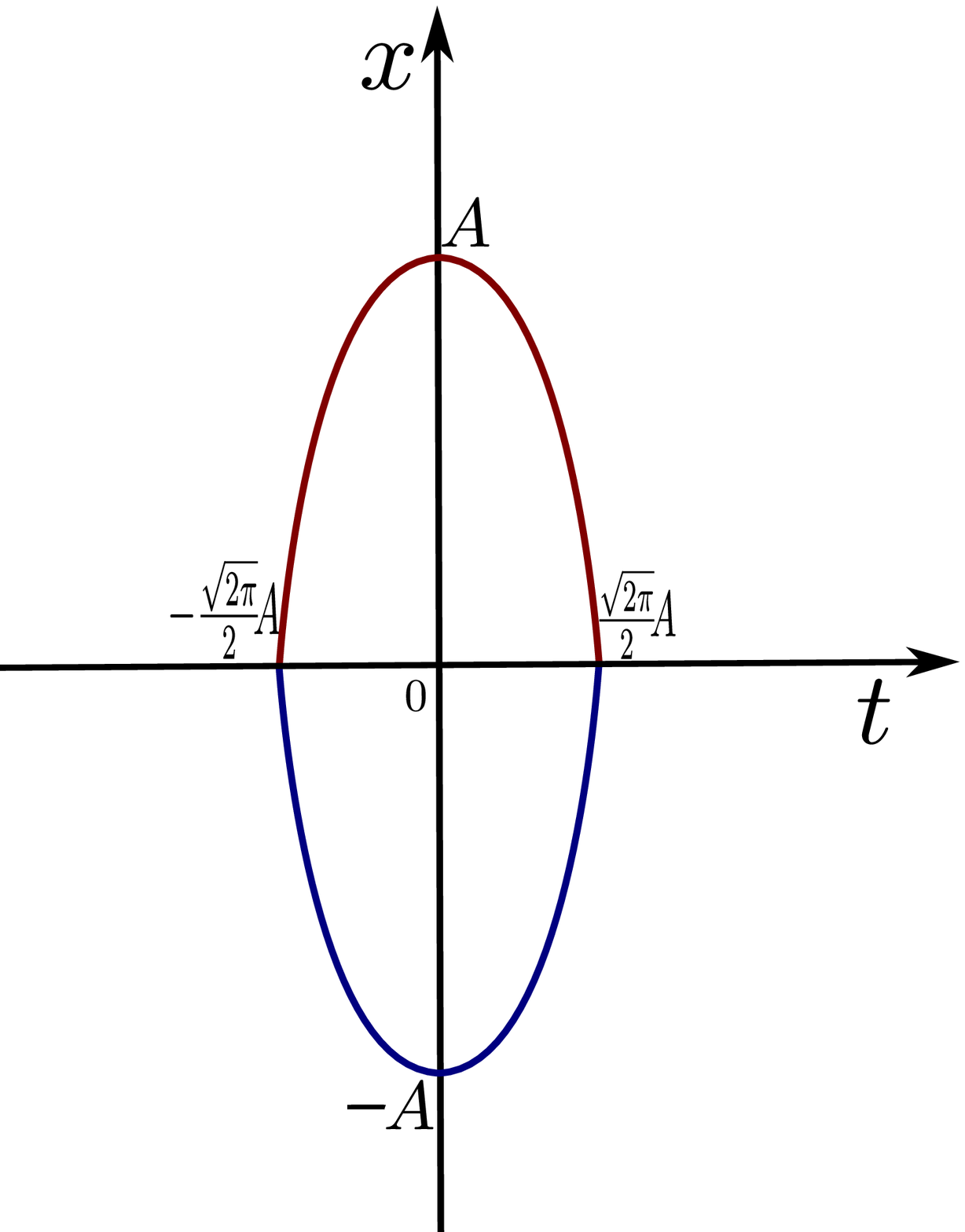,width=4.5cm}&
\qquad\qquad\epsfig{file=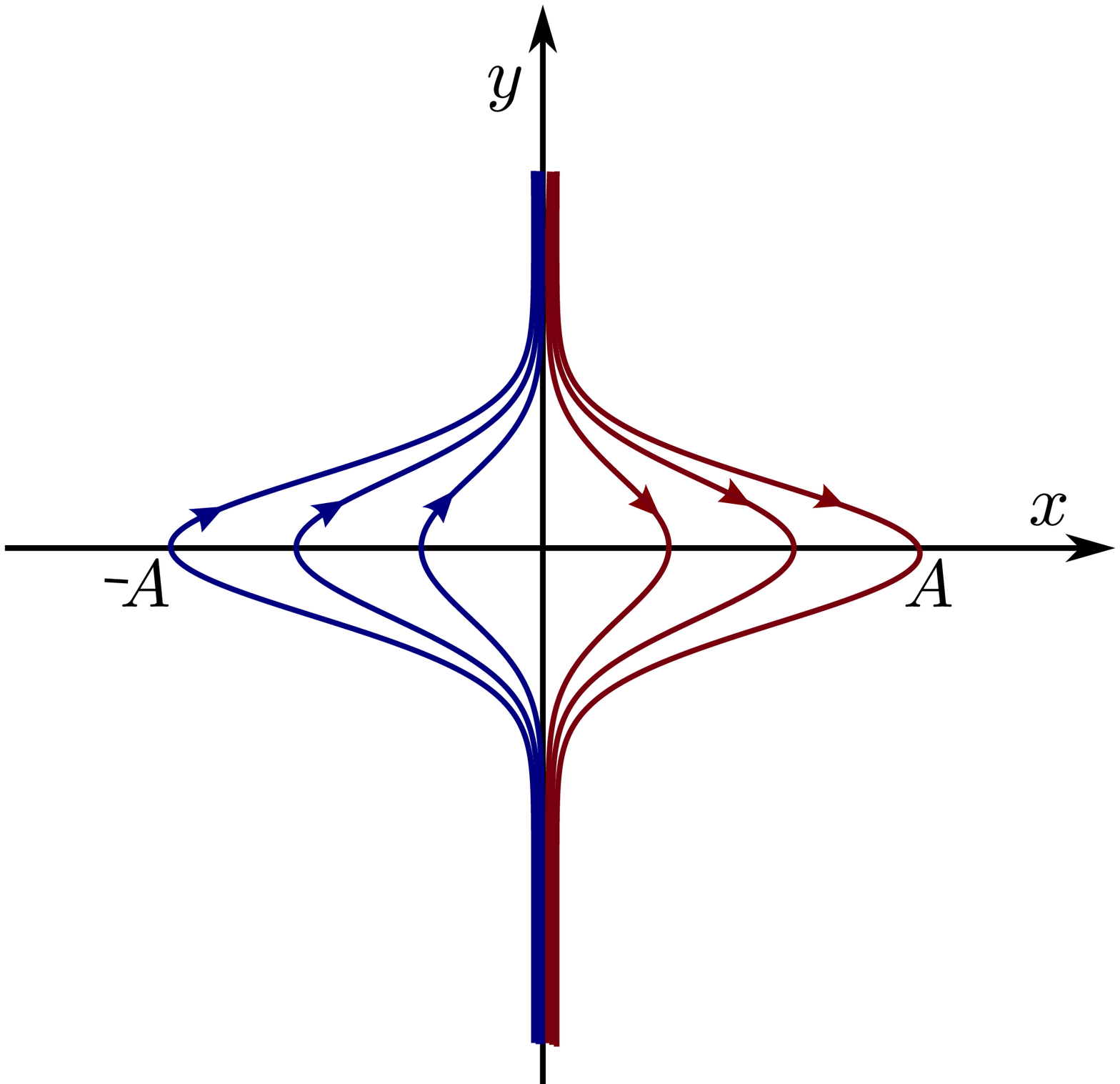,width=6cm}
\\(a)\quad&\qquad \qquad(b)
\end{tabular}
\caption{(a) Two solutions of equation \eqref{xx+1}. (b)
Phase-portrait of system~\eqref{sis1}.} \label{figura1}
\end{center}
\end{figure}

\begin{figure}[h]
\begin{center}
\begin{tabular}{c}
\epsfig{file=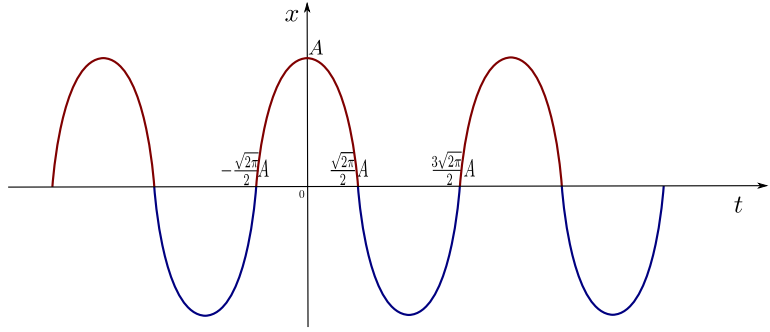,width=10cm}
\end{tabular}
\caption{A weak $\mathcal{C}^0$-periodic solution  of
\eqref{xx+1}.}\label{figura2}
\end{center}
\end{figure}

From  \eqref{solx} we introduce the ${\mathcal C}^0$-function,
defined on the whole $\R$, as
\begin{equation*}
\phi(t)=
\begin{cases}
(-1)^n\phi_0(t-n\sqrt{2}\pi), \quad \mbox{for}\quad t\in
\left(\frac{2n-1}{2}\sqrt{2\pi},\frac{2n+1}{2}\sqrt{2\pi}\right),
\quad n \in\Z,\\ \qquad\quad 0 \qquad\qquad\qquad\,\,
\mbox{for}\quad t=\frac{2n+1}{2}\sqrt{2\pi},\qquad n \in\Z,
\end{cases}
\end{equation*}
see Figure~\ref{figura2}. It is a $\mathcal{C}^0$-periodic function
of period $T(A)=2\sqrt{2\pi}A$ and $x=\phi(t)$
satisfies~\eqref{xx+1}, for all
$t\in\R\setminus\cup_{n\in\Z}\{\frac{2n+1}2\sqrt{2\pi}\}.$ Hence
$(i)$ of the theorem follows.

Notice that directly from \eqref{xx+1}, it is easy to see that this
equation can not have $\mathcal{C}^2$-solutions such that $x(t^*)=0$
for some $t^*\in\R,$ because this would imply that $\lim_{t\to t^*}
\ddot x(t)=\infty.$

\noindent $(ii)$ System~\eqref{sis2} is Hamiltonian with Hamiltonian function
$$H(x,y)=\frac{y^2}{2}+\frac{\ln(x^2+k^2)}{2}.$$
Since $\ln(x^2+k^2)$  has a  global minimum at $0$ and
$\ln(x^2+k^2)$ tends to infinity when $|x|$ does, system
\eqref{sis2} has a global center at the origin. In
Figure~\eqref{figura3} we can see its phase portrait for some values
of $k$. This figure also illustrates how the periodic orbits
of~\eqref{sis2} approach to the solutions of system~\eqref{sis1}.
\begin{figure}[h]
\begin{center}
\begin{tabular}{ccc}
\epsfig{file=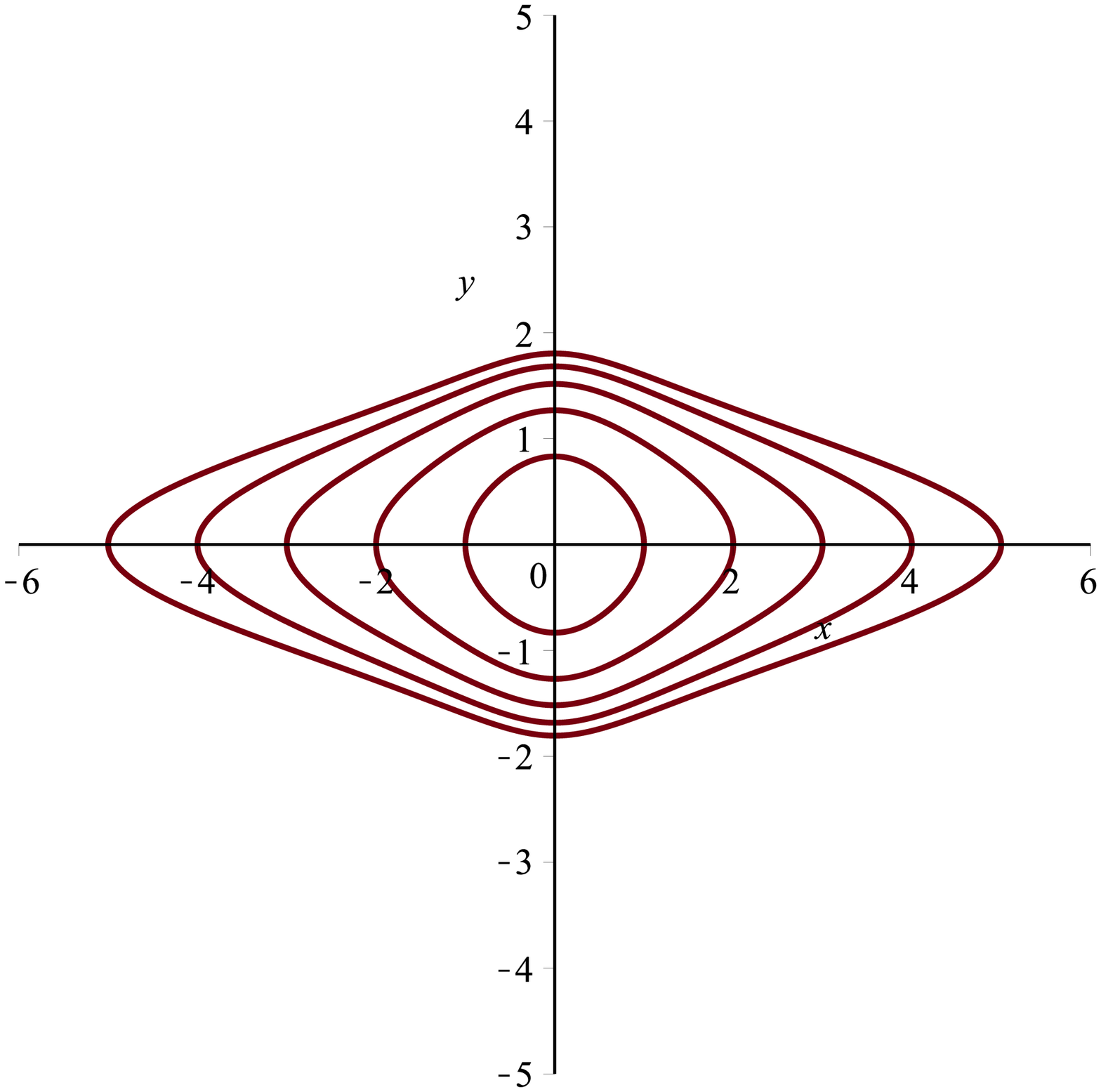,width=4.1cm}&\quad
\epsfig{file=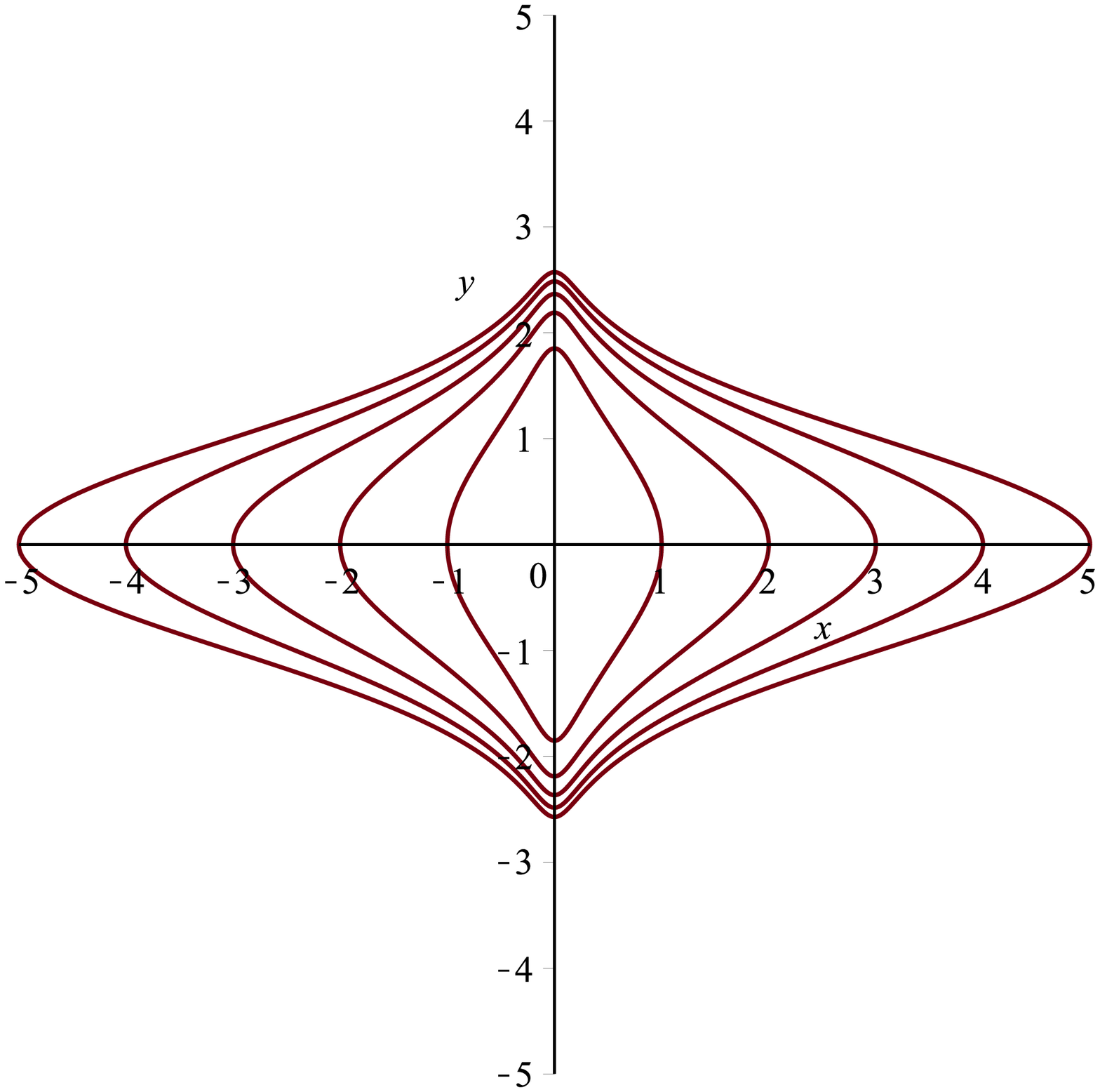,width=4.1cm}&\quad
\epsfig{file=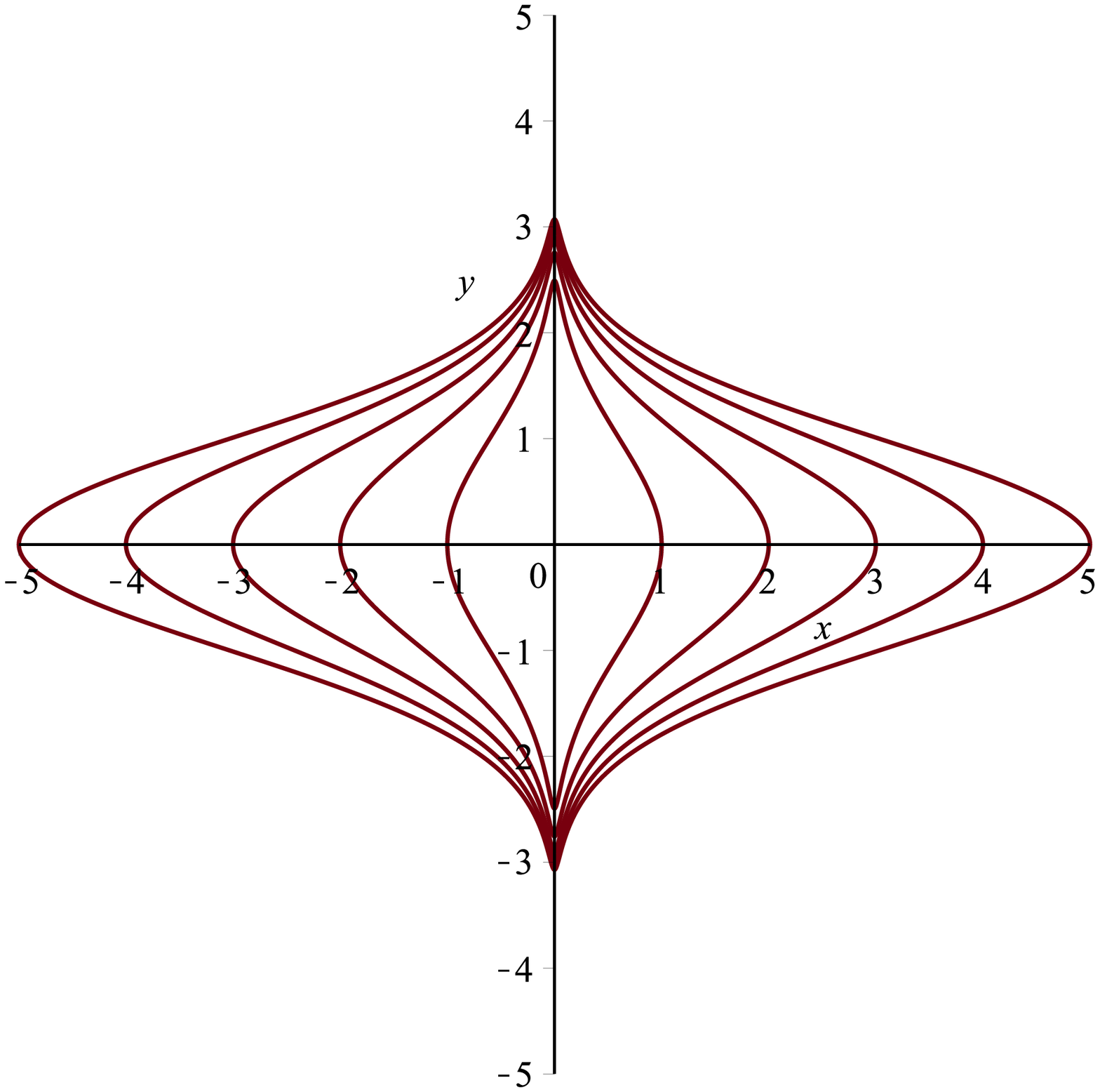,width=4.1cm}
\\$k=1$\quad&$k=\frac{1}{50}$\qquad &$k=\frac{1}{1000}$
\end{tabular}
\caption{Phase portraits of \eqref{sis2} for different  values of
$k$.} \label{figura3}
\end{center}
\end{figure}
Its period function is
$$
T(A;k)=2\int_{-A}^A \frac{dx}{y(x)}= 2\int_{-A}^A\frac{dx}{\sqrt{2h-\ln{(x^2+k^2)}}},
$$
where $h=\ln(A^2+k^2)/2$ is the energy level of the orbit passing
through the point $(A,0)$. Therefore,
$$
T(A;k)=2\int_{-A}^A\frac{ dx
}{\sqrt{\ln{\left(\frac{A^2+k^2}{x^2+k^2}\right)}}}=4\,A\int_0^1\frac{
ds}{\sqrt{\ln{\left(\frac{A^2+k^2}{A^2s^2+k^2}\right)}}},
$$
where we have used  the change of variable $s=x/A$ and  the symmetry with respect to $x.$ Then,
$$
\lim_{k\rightarrow 0}T(A;k)=\lim_{k\rightarrow 0}\int_0^1\frac{4\,A
\,ds}{\sqrt{\ln{\left(\frac{A^2+k^2}{A^2 s^2+k^2}\right)}}}.
$$
If we prove that
\begin{equation}
\label{intercanvio}
 \lim_{k\rightarrow 0}\int_0^1\frac{4\,A
\,ds}{\sqrt{\ln{\left(\frac{A^2+k^2}{A^2
s^2+k^2}\right)}}}=\int_0^1\lim_{k\rightarrow
0}\frac{4\,A \,ds}{\sqrt{\ln{\left(\frac{A^2+k^2}{A^2
s^2+k^2}\right)}}},
\end{equation}
then
$$
\lim_{k\rightarrow 0}T(A;k)=4\,A
\,\int_0^1\frac{ds}{\sqrt{-2\ln{\left({s}\right)}}}=2\sqrt{2\pi}A=T(A)
$$
and the theorem will follow. Therefore, for completing the proof, it
only remains to show that \eqref{intercanvio} holds. For proving
that, take any sequence $1/z_n,$ with $z_n$ tending monotonically to
infinity, and consider the functions
$f_n(s)=\left(\ln{\left(\frac{A^2z_n^2+1}{A^2
z_n^2s^2+1}\right)}\right)^{-1/2}.$ We have that the sequence
$\{f_n(s)\}_{n\in\N}$ is  formed by
 measurable and positive  functions defined on the interval $(0,1)$.
 It is not difficult to prove that
 it is a decreasing sequence. In particular, $f_n(s)< f_1(s)$ for all $n>1$. Therefore, if
we show that $f_1(s)$ is integrable, then we can apply the
Lebesgue's dominated convergence theorem~(\cite{Ru}) and
\eqref{intercanvio} will follow. To prove that
 $ \int_0^1f_1(s)<\infty $ note that for $s$ close to~1,
$$f_1(s)=\left(\ln{\left(\frac{A^2z_1^2+1}{A^2z_1^2s^2+1}\right)}\right)^{-1/2}
\sim\left(\frac{2A^2z_1^2(1-s)}{A^2z_1^2+1}\right)^{-1/2}.$$ Since
this last expression is integrable the result follows by the
comparison test for improper integrals.

\section{The Harmonic Balance Method}\label{hbm}

This section gives a brief description of the HBM applied  the
second order differential equations
\begin{equation}\label{im}
\mathcal{F}:=\mathcal{F}(x(t),\ddot{x}(t))=0,
\end{equation}  with $\mathcal{F}(-u,-v)=\mathcal{F}(u,v)$, and adapted to our
interests. Notice that  if $x(t)$ is a solution of \eqref{im} then $x(-t)$ also is a solution.

Suppose that equation \eqref{im}  has a $T$-periodic solution $x(t)$
with initial conditions $x(0)=A$,  $\dot{x}(0)=0$ and period
$T=T(A).$ If $x(t)$ satisfies $x(t)=x(-t)$ it is clear that its
Fourier series has no sinus terms and  writes as
\[
\sum_{k=1}^{\infty} a_k \cos(k\omega
t),\quad\mbox{with}\quad\sum_{k=1}^\infty a_k=A\quad\mbox{and}\quad
\omega=\frac{2\pi}T.
\]
As we have seen in previous section, the weak periodic solutions of
equation $x(t)\ddot x(t)+1=0$ that we want to approach satisfy the
above property. Moreover $x(T/4)=0$ and $\dot x(T/4)$ does not
exist. In any case, if we are searching smooth approximations to
this $x(t)$, they should also satisfy $\dot x(t)\ddot
x(t)+x(t)\dddot x(t)=0,$ and hence $\dot x(T/4)=0.$ For this reason,
in this work we will search Fourier series in cosinus, not having
the even terms $\cos(2j \omega t)$, $j\in\N\cup\{0\},$ which do not
satisfy this property. This type of a priori simplifications are
similar to the ones introduced in \cite{Minou} for other problems.

Hence, in our setting, the HBM of order $N$ follows the next five
steps:

1. Consider a trigonometric polynomial
\begin{equation}\label{solu}
x_{_N}(t)=\sum_{j=1}^{N} a_{2j-1} \cos((2j-1)\omega_N t)\quad
\mbox{with}\quad \sum_{j=1}^{N} a_{2j-1}=A.
\end{equation}

2. Compute the $2\pi/\omega_N$-periodic function
$\mathcal{F}_N:=\mathcal{F}(x_{_N}(t),\ddot{x}_{_N}(t))$, which has
also an associated Fourier series,
\[
\mathcal{F}_N(t)=\sum_{j\ge0} \mathcal{A}_{j} \cos(j\,\omega_N t),
\]
where $\mathcal{A}_{j}=\mathcal{A}_{j}({\mathbf a},\omega_N,A)$
$j\ge0,$ with ${\mathbf a}=(a_1,a_3,\ldots,a_{{2N-1}})$.

3. Find all values ${\bf a}$ and $\omega_N$ such that
\begin{equation}\label{sistema}
\mathcal{A}_{j}({\bf a},\omega_N,A)=0 \qquad\mbox{for}\quad 1\le
j\le j_N,
\end{equation}
where $j_N$ is the value such that \eqref{sistema} consists exactly of $N$ non trivial equations.
Notice also that each equation $\mathcal{A}_{j}({\bf a},\omega_N, A)=0$  is equivalent to
\begin{equation}\label{inte}
\int_0^{2\pi/\omega_N} \cos (j\omega_N t)\mathcal{F}_N(t)\,dt=0.
\end{equation}

4. Then the expression \eqref{solu}, with the values of ${\bf
a}={\bf a}(A)$ and $\omega_N=\omega_N(A)$  obtained in point~3,
provide candidates to be
 approximations of the actual periodic solutions of the initial
 differential equation. In particular, the functions
 $\T_N=\T_N(A)=2\pi/\omega_N$ give approximations of  the periods of the
 corresponding periodic orbits.

5. Choose, as  final approximation,  the one associated to the solution that minimizes the norm
\[
||\mathcal{F}_N(t)||=\int_0^{\T_N} \mathcal{F}_N^2(t)\,dt.
\]

\begin{remark} (i) Notice that going from order $N$ to
order $N+1$ in the method, implies to compute again all the
coefficients of the Fourier polynomial, because in general the
common Fourier coefficients of $x_{_N}(t)$ and $x_{_{N+1}}(t)$ do
not coincide.

(ii) The above set of equations \eqref{sistema} is a system of
polynomial equations which usually is not easy to solve. For this
reason in many works, see for instance \cite{Mi,Mic-book2} and the
references therein, only the values of $N=1,2$ are considered.  For
solving system \eqref{sistema} for $N\ge3$ we use the Gr\"obner
basis approach $($\cite{cox}$)$. In general this method is faster
that using successive resultants and moreover it does not give
spurious solutions.

(iii) As far as we know, the test proposed in point 5 to select the
best approach is not commonly used. We propose it following the
definition of accuracy of an approximated solution used
in~\cite{Ga-Ga} and inspired in the classical
works~\cite{Stokes,Ur}.

\end{remark}

\section{Application of the HBM}\label{sec sys}

We start proving a lemma that will allow to reduce our computations
to the case $A=1.$

\begin{lemma}  Let $\T_N(A;m)$ be the period of the truncated Fourier series
obtained applying the $N$-th order HBM to equation~\eqref{eq xm}.
Then  there exists a real constant $C_{N}(m)$  such that
$\T_N(A;m)=C_{N}(m)A$.
\end{lemma}
\begin{proof}
Consider $\mathcal{F}_N=x_N^{m+1}\ddot{x}_N+x_N^m=0$, with $x_N$ given in \eqref{solu}. We have to
solve the set of $N+1$ non-trivial equations
\begin{equation}\label{ee}
\int_0^{2\pi/\omega_N} \cos (j\omega_N t)\mathcal{F}_N(t)\,dt=0 \qquad 1\le j\le j_N,\quad \quad
\sum_{j=1}^{N} a_{2j-1}=A,
\end{equation}
with $N+1$ unknowns $a_1,a_3,\ldots,a_{{2N-1}}$ and $\omega_N$ and $A\ne0.$
 The lemma clearly follows if we prove next assertion:
$\tilde a_1,\tilde a_3,\ldots,\tilde a_{{2N-1}}$ and $\tilde
\omega_N$ is a solution of~\eqref{ee} with $A=1$ if and only if
$A\tilde a_1,A\tilde a_3,\ldots,A\tilde a_{{2N-1}}$ and $\tilde
\omega_N/A$ is  a solution of~\eqref{ee}. This equivalence is a
consequence of the fact that the change of variables $s=At$ writes
the integral equation in~\eqref{ee} as
\[\frac 1 A\int_0^{2\pi A/\omega_N}
\cos\Big(j\frac{\omega_N}A s\Big)\mathcal{F}_N\Big(\frac s
A\Big)\,ds=0
\] and from the structure of the right hand side equation of \eqref{ee}.
  Hence, $\T_N(A;m)=\T_N(1;m)A=: C_{N}(m) A,$ as we wanted to prove.
\end{proof}

\begin{proof}[Proof of Theorem~\ref{main2}]

Due to the above lemma, in the application of the $N$-th order HBM,
we can restrict our attention to the case $A=1.$

$(i)$ Following section \ref{hbm}, we consider $x_1(t)=\cos(\omega_1 t)$ as the first
approximation to the actual solution of the functional equation $\mathcal{F}(x(t),\ddot
x(t))=x^{m+1}\ddot{x}+x^m=0$. Then
$$
\mathcal{F}_1(t)=-\omega_1^2\cos^{m+2}(\omega_1 t)+\cos^{m}(\omega_1
t).
$$
When $m=2k$ the above expression writes as
\begin{equation}\label{con k}
\mathcal{F}_1(t)=-\omega_1^2\cos^{2k+2}(\omega_1
t)+\cos^{2k}(\omega_1 t)=0.
\end{equation}
Using \eqref{inte} for $j=0$ we get
\begin{equation}\label{con}
\int_0^{2\pi/\omega_1} \mathcal{F}_1(t)\,dt=
-\omega_1^2{I}_{2k+2}+{I}_{2k}= 0,
\end{equation}
where ${I}_{2\ell}=\int_0^{2\pi/\omega_1}\cos^{2\ell}(\omega_1 t)
dt$. By using integration by parts we prove that $
(2k+2){I}_{2k+2}=(2k+1){I}_{2k}. $ Combining this equality and
\eqref{con} we obtain that
$$
\omega_1=\sqrt{\frac{2k+2}{2k+1}},
$$
or equivalently,
$$
\T_1(A;m)=2\pi A\sqrt{\frac{2k+1}{2k+2}},
$$
that in terms of $m$ coincides with \eqref{PFfmic Sin}. The case $m$
odd follows similarly. The only difference is that instead of
condition \eqref{con} to find $\T_1(A;m)$ we have to impose that
\begin{equation*}
\int_0^{2\pi/\omega_1} \cos(\omega_1 t)\mathcal{F}_1(t)\,dt= 0,
\end{equation*}
because $\int_0^{2\pi/\omega_1} \mathcal{F}_1(t)\,dt\equiv0.$

\smallskip
\smallskip

\noindent $(ii)$ Case $m=0$. Consider the functional equation
$\mathcal{F}(x(t),\ddot{x}(t))=x(t)\ddot{x}(t)+1=0.$

When $N=2$, we take an approximation $ x_2(t)=a_1\cos(\omega_2
t)+a_3\cos(3\omega_2 t). $ The vanishing of the coefficients of 1
and $\cos(2\omega_2 t)$ in the Fourier series of $\mathcal{F}_2$
provides the nonlinear system
\begin{align*}
1-\frac{1}{2}(a_1^2+9 a_3^2)\omega_2^2&=0,\\
a_1+10a_3&=0,\\
a_1+a_3-1&=0.
\end{align*}
By solving it and applying point 5 in the HBM we  get that $
\omega_2=18/\sqrt{218}. $ Therefore,
\[\T_2(A;0)=\frac{\sqrt{218}}9\pi A\approx 5.1539 A, \] as we wanted
to prove.
\smallskip

\noindent  For the third-order HBM we use as approximate solution $
x_3(t)=a_1\cos(\omega_3 t)+a_3\cos(3\omega_3 t)+a_5\cos(5\omega_3
t). $ Imposing that the coefficients of 1, $\cos(2\omega_3 t)$, and
$\cos(4\omega_3 t)$ in $\mathcal{F}_3$ vanish we arrive at the
system
\begin{align*}
 P&=2-\left(a_1^2+9a_3^2+25a_5^2\right)\omega_3^2=0,\\
 Q&=a_1^2+10a_1a_3+34a_3a_5=0, \\
 R&=5a_3+13a_5=0,\\
 S&=a_1+a_3+a_5-1=0.
\end{align*}
Since all the equations are polynomial, the searching of its
solutions can be done by using the Gr\"obner basis approach,
see~\cite{cox}. Recall that the idea of this method consists in
finding a new systems of generators, say $G_1,G_2,\ldots, G_\ell,$
of the ideal of $\R[a_1,a_3,a_5,\omega_3]$ generated by $P,Q,R$ and
$S$. Hence,
 solving  $P=Q=R=S=0$ is equivalent to solve
 $G_i=0,$ $i=1,\ldots,\ell$. In general, choosing the lexicographic
 ordering in the  Gr\"obner basis
 approach, we get that the polynomials of the equivalent system have   triangular
structure with respect to the variables and it can  be easily
solved.

Now, by computing the Gr\"{o}bner basis of $\{P,Q,R,S\}$ with respect to
the lexicographic ordering $[a_1,a_3,a_5,\omega_3]$ we obtain a new
basis with  four polynomials ($\ell=4$), being one of them,
$$
G_1(\omega_3)=1553685075\omega_3^8-3692301106\omega_3^6+2143547654\omega_3^4-402413472\omega_3^2+20301192.
$$

Solving $G_1(\omega_3)=0$ and using again point 5 of our approach to
HBM we get that the solution that gives the better approximation is
$$
\omega_3=\frac{3\sqrt{5494790257313+115642506449\sqrt{715}}}{6905267}.
$$
Hence the expression $\T_3(A;0)=2\pi A/\omega_3$ of the statement follows.

\smallskip

When $N=4$ we consider
$
x_4(t)=a_1\cos(\omega_4 t)+a_3\cos(3\omega_4 t)+a_5\cos(5\omega_4
t)+a_7\cos(7\omega_4 t),
$
and  we arrive at the system
\begin{align*}
 P&=2-\left(a_1^2+9a_3^2+25a_5^2+49a_7^2\right)\,\omega_4^2=0,\\
 Q&=a_1^2+10a_1a_3+34a_3a_5+74a_5a_7=0,\\
 R&=5a_1a_3+13a_1a_5+29a_3a_7=0,\\
 S&=9a_3^2+50a_1a_7+26a_1a_5=0,\\
 U&=a_1+a_3+a_5+a_7-1=0.
\end{align*}
The Gr\"{o}bner basis of $\{P,Q,R,S,U\}$ with respect to the
lexicographic ordering $[a_1,a_3,a_5,a_7,\omega_4]$ is a new basis
with  five polynomials, being one of them an even polynomial in
$\omega_4$ of degree 16 with integers coefficients. Solving it we
obtain that the best approximation is $ \omega_4\approx {1.2453}, $
which gives $\T_4(A;0)\approx 5.0455\,A.$

For $N=5$ and $N=6$ we have done similar computations. In the case
$N=5$ one of the generators  of the Gr\"obner basis  is an even
polynomial in $\omega_5$ with integers coefficients and  degree 32.
When $N=6$ the same happens but with a polynomial of degree 64 in
$\omega_6$. Solving the corresponding polynomials we get that
$\omega_5\approx1.2606$ and $\omega_6\approx 1.2501$, and
consequently, $\T_5(A;0)\approx 4.9843\,A,$ and $\T_6(A;0)\approx
5.0260\,A.$

\smallskip
\smallskip

\noindent $(iii)$ Case $m=1$. We apply  the HBM to
$\mathcal{F}(x(t),\ddot{x}(t))=x^2(t)\ddot{x}(t)+x(t)=0.$

When $N=2,$ doing similar computations that in  item $(ii)$, we
arrive at
\begin{align*}
P&= 4-\left(3a_1^2+11a_1a_3+38a_3^2\right)\,\omega_2^2=0,\\
Q&= 4a_3-\left(a_1^3+22a_1^2a_3+27a_3^3\right)\omega_2^2=0,\\
R&=a_1+a_3-1=0.
\end{align*}
Again, by searching the Gr\"{o}bner basis of $\{P,Q,R\}$ with
respect to the lexicographic ordering $[a_1,a_3,\omega_2]$ we obtain
a new basis with three polynomials, being one of them
\begin{align*}
G_1(\omega_2)=7635411\omega_2^8-14625556\omega_2^6+5833600\omega_2^4-661376\omega_2^2+13824.
\end{align*}
 Notice that the equation $G(\omega_2)=0$ can be algebraically solved.
Nevertheless, for the sake of shortness, we do not give the exact
roots. Following again step 5 of our approach we get that the best
solution is $\omega_2\approx1.1915$, or equivalently that $
\T_2(A;1)\approx 5.2733 A.$

\smallskip

The HBM when  $N=3$ produces the system
\begin{align*}
P &=
4a_1-\left(3a_1^3+11a_1^2a_3+38a_1a_3^2+70a_1a_3a_5+102a_1a_5^2+43a_3^2a_5\right)\omega_3^2=0,\\
Q &= 4a_3-\left(a_1^3+22a_1^2a_3+27a_1^2a_5+70a_1a_3a_5 +27
a_3^3\right)\omega_3^2=0,\\
R &= 4a_5-\left(11a_1^2 a_3+54a_1^2 a_5+19a_1a_3^2 +86 a_3^2
a_5+75a_5^3\right)\omega_3^2=0,\\
S&=a_1+a_3+a_5-1=0.
\end{align*}

Computing the Gr\"{o}bner basis of $\{P,Q,R,S\}$ with respect to the
lexicographic ordering $[a_1,a_3,a_5,\omega_3]$ we get that one of
the polynomials of the new basis is an even polynomial in $\omega_3$
of degree 26 with integer coefficients. By solving it we obtain that
the best approximation is $\omega_3\approx 1.2206$, which produces
the value $\T_3(A;1)$ of the statement.

\smallskip

When $N=4$ we arrive at five polynomial equations, that we omit.
Once more, using the Gr\"{o}bner basis approach  we obtain a
polynomial condition in $\omega_4$ of degree 80. Finally, $
\omega_4\approx 1.2275 $ and $\T_4(1;A)\approx 5.1186.$

\smallskip
\smallskip

\noindent $(iv)$ When $m=2$ we have to  approach the solutions of
$\mathcal{F}(x(t),\ddot{x}(t))=x^3(t)\ddot{x}(t)+x^2(t)=0.$ We do
not give  the details of the proof because we get our results by
using exactly the same type of computations.

\end{proof}

\begin{remark}
For each $N$ and $m$ our computations also provide a trigonometric
polynomial that approaches the continuous weak periodic solution
$\phi(t)$ given in the proof of Theorem \ref{main1}.
\end{remark}
\subsection*{Acknowledgements}
The two authors are  supported by the MICIIN/FEDER  grant number
MTM2008-03437, FEDER-UNAB10-4E-378 and the Generalitat de Catalunya
grant number 2009-SGR 410. The first author is also supported by the
grant FPU AP2009-1189.

\end{document}